\documentclass[a4paper,11pt]{amsart}

\usepackage[a4paper,left=2.5cm,right=2.5cm,top=2.5cm,bottom=2.5cm]{geometry}

\usepackage{amsfonts}
\usepackage{amsmath}
\usepackage{amssymb}
\usepackage{amsthm}

\usepackage{mathrsfs}

\DeclareFontFamily{U}{mathx}{\hyphenchar\font45}
\DeclareFontShape{U}{mathx}{m}{n}{
      <5> <6> <7> <8> <9> <10>
      <10.95> <12> <14.4> <17.28> <20.74> <24.88>
      mathx10
      }{}
\DeclareSymbolFont{mathx}{U}{mathx}{m}{n}
\DeclareFontSubstitution{U}{mathx}{m}{n}
\DeclareMathAccent{\widecheck}{0}{mathx}{"71}

\usepackage{paralist}
\usepackage[colorlinks,cite color=blue,pagebackref=true,pdftex]{hyperref}
\numberwithin{equation}{section}

\usepackage{cleveref}

\usepackage[T1]{fontenc}

\usepackage{caption}
\usepackage{subcaption}

\usepackage{tikz}
\usepackage{tikz-3dplot}
\usetikzlibrary{calc,through,backgrounds,shapes,matrix,angles}
\usetikzlibrary{positioning}

\newcommand{\ie}{{i.e.}} %
\newcommand{\hosymbol}{
\begin{tikzpicture}[scale=1.7]
\draw[very thin] (0,0) circle (0.05cm);
\filldraw[black] (0,0)--(0,-0.05cm) arc (-90:90:0.05cm) -- (0,0);
\end{tikzpicture}
}

\newcommand{\covcirc}{
\begin{tikzpicture}[scale=.3,baseline=-0.5ex]
    \draw[fill=black] (0,0) circle (.6ex);
\end{tikzpicture}
}

\newcommand{\Deltaprime}{{\Delta^{\!\hosymbol}}}

\newcommand\Pint[1]{{[a,b]_\P}}

\newcommand\Defn[1]{\textbf{\color{black}#1}}

\newcommand\hopen{\mathbb{H}}%

\newcommand\parents{\mathcal{N}}

\renewcommand\v{\mathbf{v}}
\newcommand\p{\mathbf{p}}
\newcommand\q{\mathbf{q}}
\renewcommand\r{\mathbf{r}}
\newcommand\Ptop{\widehat{1}}
\newcommand\Pbot{{\widehat{0}}}
\renewcommand\emptyset{\varnothing}
\newcommand\Z{\mathbb{Z}}               
\newcommand\Znn{\mathbb{Z}_{\ge0}}               
\newcommand\R{\mathbb{R}}               

\renewcommand\c{\mathbf{c}}
\newcommand\w{\mathbf{w}}
\newcommand\x{\mathbf{x}}
\newcommand\y{\mathbf{y}}

\newcommand\SymGrp{\mathfrak{S}}
\renewcommand\P{P}
\newcommand\Phat{{\widecheck{\P}}}

\newcommand\Po{\mathcal{P}}

\newcommand\0{\mathbf{0}}
\newcommand\1{\mathbf{1}}

\DeclareMathOperator{\stat}{Ldes}
\DeclareMathOperator{\GenFun}{\mathcal{G}}

\DeclareMathOperator{\Des}{Des}
\DeclareMathOperator{\des}{des}
\DeclareMathOperator{\iDes}{iDes}
\DeclareMathOperator{\ides}{ides}

\DeclareMathOperator{\bigasc}{asc^{(2)}}
\DeclareMathOperator{\bigiasc}{iasc^{(2)}}

\DeclareMathOperator{\Int}{int}
\DeclareMathOperator{\conv}{conv}
\DeclareMathOperator{\vol}{vol}
\DeclareMathOperator{\aff}{aff}

\newcommand{\tauinv}{{\hat\tau}}

\newcommand\Lip{\mathscr{L}}%
\newcommand\cov{{\,\prec\!\!\covcirc\,}}

\newcommand\eqdef{ :=  }

\newcommand\eq{\ = \ }
\newcommand\newmid{\ : \ }

\newtheorem{thm}{Theorem}[section]
\newtheorem{cor}[thm]{Corollary}
\newtheorem{lem}[thm]{Lemma}
\newtheorem{prop}[thm]{Proposition}
\newtheorem{conj}[thm]{Conjecture}

\theoremstyle{definition}

\newtheorem{example}[thm]{Example}

\newtheorem{openproblem}[thm]{Open problem}

\title{Lipschitz polytopes of posets and permutation statistics}

\author{Raman Sanyal}
\address{Institut f\"ur Mathematik, Goethe-Universit\"at Frankfurt, Germany}
\email{sanyal@math.uni-frankfurt.de}

\author{Christian Stump}
\address{Institut f\"ur Mathematik, Freie Universit\"at Berlin, Germany}
\email{christian.stump@fu-berlin.de}
\thanks{R.~Sanyal was supported by the DFG-Collaborative Research Center, TRR
109 ``Discretization in Geometry and Dynamics''. C.~Stump was supported by the
DFG grant STU 563/2 ``Coxeter-Catalan combinatorics''.  }

\keywords{posets, isotone functions, Lipschitz polytopes,
lattice polytopes, Gorenstein polytopes, permutation statistics}

\subjclass[2010]{
06A07, %
05A05, %
52B12} %

\date{\today}

\begin{document}

\begin{abstract}
    We introduce Lipschitz functions on a finite partially
    ordered set $\P$ and study the associated Lipschitz
    polytope $\Lip(\P)$. The geometry of $\Lip(\P)$ can be
    described in terms of descent-compatible permutations
    and permutation statistics that generalize descents and
    big ascents. For ranked posets, Lipschitz polytopes are
    centrally-symmetric and Gorenstein, which implies
    symmetry and unimodality of the statistics. Finally, we
    define $(\P,k)$-hypersimplices as generalizations of
    classical hypersimplices and give combinatorial
    interpretations of their volumes and $h^*$-vectors.
\end{abstract}

\maketitle

\section{Introduction}\label{sec:intro}

\newcommand\OK{\mathcal{K}}%
\newcommand\OC{\mathcal{C}}%
\newcommand\OP{\mathcal{O}}%
Let $(\P,\preceq)$ a finite partially ordered set (or \Defn{poset}, for
short).  A function $f : \P \rightarrow \R$ is \Defn{isotone} or \Defn{order
preserving} if
\[
    f(a) \ \le \ f(b) \quad \text{whenever} \quad a \ \preceq \ b \,.
\]
The \Defn{order cone} $\OK(\P)$ of~$\P$ is the collection of nonnegative
isotone functions. The order cone is a gateway for a geometric perspective on
enumerative problems on posets. The interplay of combinatorics and geometry is,
in particular, fueled by analogies to \emph{continues} mathematics. For
example, Stanley's order polytope~\cite{TwoPoset} is the set
\[
    \OP(\P) \ = \ \big\{ f \in \OK(\P) \newmid \|f\|_\infty \le 1 \big\} \,,
\]
where $\|f\|_\infty = \max \{ f(a) : a \in \P \}$. The theory of~$\P$-partitions
concerns those $f \in \OK(\P)$ with $\|f\|_1 = \sum_a f(a) = m$ for some
fixed~$m$.  In this paper, we want to further the analogies to continuous
functions.  For two elements $a, b \in \P$, we denote the minimal length of a
saturated (or unrefineable) chain from~$a$ to~$b$ by~$d_\P(a,b)$ and
set $d_\P(a,b) := \infty$ if $a \not\preceq b$. Then $d_\P$ is a
\emph{quasi-metric} on $\P$. An isotone function $f : (\P,\preceq) \to \R$ is
\Defn{$\boldsymbol{k}$-Lipschitz} if 
\[
    f(b) - f(a) \ \le \ k \cdot d_\P(a,b)
\]
for all $a \preceq b$. We say that a function $f$ is Lipschitz if $f$ is
$1$-Lipschitz. Let us write $\Phat$ for the poset obtained from $\P$ by
adjoining a minimum $\Pbot$. The collection $\widetilde{\Lip}(\Phat)$ of
isotone Lipschitz functions on $\Phat$ is naturally an unbounded polyhedron and
$k$-Lipschitz functions are precisely the elements in $k \cdot
\widetilde{\Lip}(\Phat)$. The lineality space of $\widetilde{\Lip}(\Phat)$ is
given by all constant functions and we define the \Defn{Lipschitz
polytope} of~$\P$ as 
\[
    \Lip(\P) \ \eqdef \ \big\{ f \in \OK(\Phat) \newmid f \text{ Lipschitz},
    \  f(\Pbot) =
    0 \big\}\,.
\]
Concretely, the Lipschitz polytope of $(\P,\preceq)$ is given by
\begin{equation}\label{eqn:Lip}
    \Lip(\P) \ = \ \left\{ f \in \R^\P :
    \begin{array}{c@{\ \le \ }c@{\ \le \ }cl}
        0 & f(a) & 1 & \text{ for } a \in \min \P \\
        0 & f(b) - f(a) & 1 & \text{ for } a \cov b\\
    \end{array}
    \right\},
\end{equation}
where $a \cov b$ denotes the cover relations of $\P$.

A different motivation for the study of $\Lip(\P)$ comes from $G$-Shi
arrangements. The Hasse diagram of~$\Phat$ is the directed graph $G$ on nodes
$\Phat$ with arcs $(a,b)$ whenever $a \cov b$ is a cover relation. The
corresponding \Defn{$\boldsymbol G$-Shi arrangement} is the arrangement of
affine hyperplanes $\{ x_b - x_a = 0\}$ and $\{x_b - x_a = 1\}$ for $a \cov
b$.  The $G$-Shi arrangements generalize the classical Shi
arrangements~\cite[Ch.~7]{Shi} and naturally occur in the geometric
combinatorics of parking functions and spanning trees; see~\cite{HP}.  The
Lipschitz polytope $\Lip(\P)$ is thus a particular (relatively) bounded region
of the $G$-Shi arrangement associated to the Hasse diagram of~$\P$.

We give some basic geometric properties of Lipschitz polytopes in
\Cref{sec:basics} and, in particular, show that $\Lip(\P)$ is always a lattice
polytope. Hence, the function
\[
    E(\Lip(\P),k) \ := \  \bigl|\Lip(\P) \cap \Z^\P \bigr| 
\]
counting integer-valued $k$-Lipschitz functions agrees with a polynomial of
degree $|\P|$. In \Cref{sec:triangulations}, we describe the canonical regular
and unimodular triangulation of $\Lip(\P)$. Similar to the study of order
polytopes, the simplices of the triangulation of $\Lip(\P)$ can be described
in terms of certain permutations of $\P$. A \emph{descent-compatible}
permutation  is a labeling of~$\P$ such that the number of descents along any
saturated increasing chain with fixed endpoints is constant. The $h^*$-vector
of $\Lip(\P)$ can be determined in terms of the combinatorics of
$(\P,\preceq)$ and defines an ascent-type statistic on the collection of all
descent-compatible permutations of~$\P$ If the Hasse diagram of~$\P$ is a
rooted tree, then all permutations are descent-compatible and the statistic
corresponds to \emph{big ascents}, i.e.\ ascents that with step
size at least~$2$.

For posets $\P$ such that $\Phat$ is ranked, the Lipschitz polytope $\Lip(\P)$
is centrally-symmetric and Gorenstein. Using results from the theory of
lattice polytopes, we deduce in \Cref{sec:ranked} that the statistic on
descent-compatible permutations is symmetric and unimodal, similar to the
classical descent statistic on all permutations.

For posets $\P$ with a unique maximal element $\Ptop$, we define in
\Cref{sec:Hyp} the $(\P,k)$-hypersimplices. These are certain slaps of
$\Lip(\P)$ that generalize the well-known $(n,k)$-hypersimplices
of~\cite{GGMS}. In particular, the volumes of our $(\P,k)$-hypersimplices can
be interpreted as the number of descent-compatible permutations of $\P$ with $k$
$P$-descents. For a chain, this recovers the classical interpretation of
volumes of $(n,k)$-hypersimplices as Eulerian numbers. In~\cite{NanLi}, Li
gave an interpretation of the coefficient of $h^*$-vector of \emph{half-open}
$(n,k)$-hypersimplices as the exceedence statistic on permutations with $k$
descents and our approach to the $h^*$-vectors of Lipschitz polytopes via
half-open decompositions yields a simple geometric proof of this fact.

\section{Basic geometric properties}\label{sec:basics}

In this section, we collect basic geometric properties of Lipschitz polytopes.
To begin with, we put some examples on record.

\begin{example}[Antichains and chains]
\label{ex:trivial}
  If $\P$ is an antichain, then $\Lip(\P) = [0,1]^\P$.
  On the other hand, if $\P = [n] := \{1,\dots,n\}$ is a chain, then
  $\Lip(\P)$ is also linearly isomorphic to $[0,1]^\P$ under the map that
  takes $g \in [0,1]^\P$ to the function $f : [n] \rightarrow \R$ with $f(i) =
  g(1) + \cdots + g(i)$.  
\end{example}

Such a linear and lattice-preserving transformation from $\Lip(\P)$ to a cube
as in the previous example exists in other cases as well. Let us call a poset
$\P$ a \Defn{rooted tree} if $\P$ has a unique minimal element and the Hasse
diagram is a tree. For a connected poset, \ie, a poset for which the Hasse diagram is connected, this
is equivalent to the property that for any $b \in \P$, there is a unique
$\bar{b} \in \Phat$ with $\bar{b} \cov b$. Define $T_\P : \R^\P \to \R^\P$
with $(Tf)(b) := f(b) - f(\bar{b})$.  This is an invertible linear and
$\Z^P$-preserving transformation.

\begin{prop}\label{prop:rooted_cube}
    Let $\P$ be a rooted tree. Then $T(\Lip(\P))  =  [0,1]^\P$.
\end{prop}
\begin{proof}
    It follows from~\eqref{eqn:Lip} that $0 \le Tf(b) \le 1$ for all $b \in
    \P$ and hence $T(\Lip(\P)) \subseteq [0,1]^\P$. For any $b \in \P$, there
    is a unique maximal chain $\Pbot = c_1 \cov \cdots \cov c_k = b$ and for
    $g \in \R^\P$, it is easy to see that $(T^{-1}g)(b) = g(c_1) + \cdots +
    g(c_k)$. Thus, $(T^{-1}g)(\bar{b}) = g(c_1) + \cdots + g(c_{k-1})$ and
    $T^{-1}(g) \in \Lip(\P)$ whenever $0 \le g(b) \le 1$ for all $b \in \P$.
    This proves the claim.
\end{proof}

\begin{example}\label{ex:nontriv}
    The figure to the right, produced with the \texttt{TikZ} output of 
    \texttt{SageMath}~\cite{sagemath}, 
    illustrates $\Lip(\P)$ for the following poset on $3$ elements:
  \begin{center}
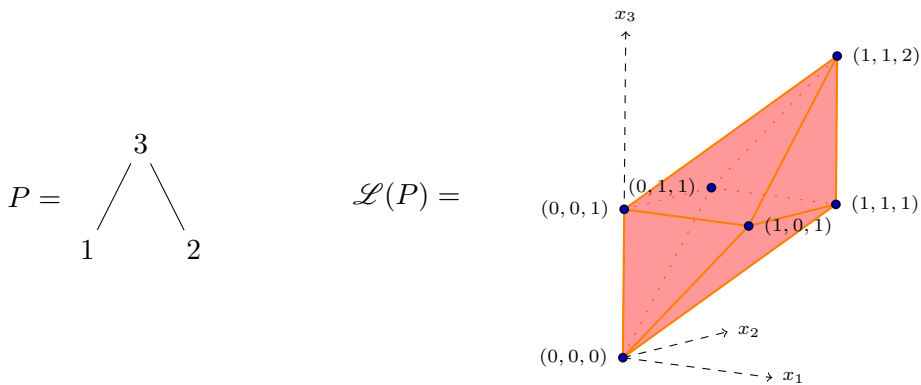
  \begin{minipage}{.4\textwidth}
    \centering
    \begin{tikzpicture}[scale=.7]
      \node (lab1) at (-2,1) {$\P =$};
      \node (1) at (-1,0) {$1$};
      \node (2) at ( 1,0) {$2$};
      \node (3) at ( 0,2) {$3$};
      \draw (1) -- (3);
      \draw (2) -- (3);
      \node (lab2) at (5,1) {$\Lip(\P) =$};
    \end{tikzpicture}
  \end{minipage}
  \begin{minipage}{.4\textwidth}
    \centering
    \begin{tikzpicture}%
      [x={(0.818921cm, -0.110558cm)},
      y={(0.573844cm, 0.143289cm)},
      z={(0.008452cm, 0.983486cm)},
      scale=2,
      back/.style={loosely dotted, thin},
      edge/.style={color=orange, thick},
      facet/.style={fill=red,fill opacity=0.400000},
      vertex/.style={inner sep=1pt,circle,draw=blue!25!black,fill=blue!75!black,thick}
      ]
    \draw[color=black,dashed,->] (0,0,0) -- (1.2,0,0) node[anchor=west]{\tiny$x_1$};
    \draw[color=black,dashed,->] (0,0,0) -- (0,1.2,0) node[anchor=west]{\tiny$x_2$};
    \draw[color=black,dashed,->] (0,0,0) -- (0,0,2.2) node[anchor=south]{\tiny$x_3$};
    \coordinate (0, 0, 0) at (0, 0, 0);
    \coordinate (0, 1, 1) at (0, 1, 1);
    \coordinate (0, 0, 1) at (0, 0, 1);
    \coordinate (1, 1, 2) at (1, 1, 2);
    \coordinate (1, 0, 1) at (1, 0, 1);
    \coordinate (1, 1, 1) at (1, 1, 1);
    \draw[edge,back] (0, 0, 0) -- (0, 1, 1);
    \draw[edge,back] (0, 1, 1) -- (0, 0, 1);
    \draw[edge,back] (0, 1, 1) -- (1, 1, 2);
    \draw[edge,back] (0, 1, 1) -- (1, 1, 1);
    \fill[facet] (1, 1, 1) -- (1, 1, 2) -- (1, 0, 1) -- cycle {};
    \fill[facet] (1, 1, 1) -- (0, 0, 0) -- (1, 0, 1) -- cycle {};
    \fill[facet] (1, 0, 1) -- (0, 0, 0) -- (0, 0, 1) -- cycle {};
    \fill[facet] (1, 0, 1) -- (0, 0, 1) -- (1, 1, 2) -- cycle {};
    \draw[edge] (0, 0, 0) -- (0, 0, 1);
    \draw[edge] (0, 0, 0) -- (1, 0, 1);
    \draw[edge] (0, 0, 0) -- (1, 1, 1);
    \draw[edge] (0, 0, 1) -- (1, 1, 2);
    \draw[edge] (0, 0, 1) -- (1, 0, 1);
    \draw[edge] (1, 1, 2) -- (1, 0, 1);
    \draw[edge] (1, 1, 2) -- (1, 1, 1);
    \draw[edge] (1, 0, 1) -- (1, 1, 1);
    \node[vertex,label=left:{\tiny$(0,1,1)$}] at (0, 1, 1)     {};
    \node[vertex,label=left:{\tiny$(0,0,0)$}] at (0, 0, 0)     {};
    \node[vertex,label=left:{\tiny$(0,0,1)$}] at (0, 0, 1)     {};
    \node[vertex,label=right:{\tiny$(1,1,2)$}] at (1, 1, 2)     {};
    \node[vertex,label=right:{\tiny$(1,0,1)$}] at (1, 0, 1)     {};
    \node[vertex,label=right:{\tiny$(1,1,1)$}] at (1, 1, 1)     {};
    \end{tikzpicture}
  \end{minipage}
  \end{center}
\end{example}

The following result is immediate and we will henceforth assume that all
posets are connected.
\begin{prop}
    If $\P = \P_1 \uplus \P_2$, then $\Lip(\P) = \Lip(\P_1) \times
    \Lip(\P_2)$. 
\end{prop}

Recall that $F \subseteq \P$ is a \Defn{filter} if $a \in F$ and $a \preceq b$,
implies $b \in F$. For a filter $F \subseteq \P$, we denote by 
\[
    \parents(F) \ \eqdef \ \big\{ a \in \Phat \setminus F \newmid b \in F \text{
    for some } a \cov b \big\} \ \subseteq \ \Phat \,.
\]
the \Defn{neighborhood} of~$F$. In particular, $\Pbot \in \parents(F)$
whenever $F$ contains a minimal element of $\P$.
A chain $\emptyset \neq F_m \subset \cdots
\subset F_1 \subseteq \P$ of filters is \Defn{neighbor-closed} if $F_{i+1}
\cup \parents(F_{i+1}) \subseteq F_i$, or, equivalently, if
\[
    a \not\in F_i \quad \Longrightarrow \quad b \not\in F_{i+1}
\]
for all $a,b \in \Phat$ with $a \cov b$ and $1 \le i < m$.
With this notion, we have the following description of the vertices of Lipschitz polytopes.

\begin{prop}\label{prop:vertices}
    Let $\P$ be a finite poset. Then
  $\v  \in \Z^\P$ is a vertex of $\Lip(\P)$ if and only if
  \[
      \v \ = \  \1_{F_m} + \cdots + \1_{F_1}
  \]
  for a neighbor-closed chain $F_m \subset \cdots \subset F_1
  \subseteq \P$ of nonempty filters in~$\P$.
\end{prop}
\begin{proof}
    Let us first observe that if $f \in \Lip(\P) \cap \Z^\P$, then $f(b) -
    f(a) = 1$ or $f(b) - f(a) = 0$ for all $\Pbot \preceq a \cov b$.  In
    particular, this means that every lattice point can be uniquely recovered
    from the knowledge of which defining linear inequalities are satisfied
    with equality and hence the the vertices are the only lattice points of
    $\Lip(\P)$.  It therefore suffices to show that every lattice point of
    $\Lip(\P)$ corresponds to a unique chain of filters as stated.  This is
    quite standard: Let $\{t_1 < \cdots < t_m\}$ be the distinct values of $f$
    and define $F_i \eqdef \{ a \in \P : f(a) \ge t_i \}$.  Since $f$ is
    isotone, $F_m \subset \cdots \subset F_1$ is a chain of nonempty filters
    and
    \[
        f \ = \ \1_{F_m} + \cdots + \1_{F_1} \, .
    \]
    Now, $f \in \Lip(\P)$ if and only if $f(b) - f(a) \le 1$ for all $a \cov
    b$. That is, if there is at most one $F_i$ with $b \in F_i$ and $a \not\in
    F_i$. This proves the claim.
\end{proof}

\begin{example}
\label{ex:shorttrees}
    The Lipschitz polytopes of the \emph{short rooted tree} on the left and
    the \emph{short hanging tree} on the right
  
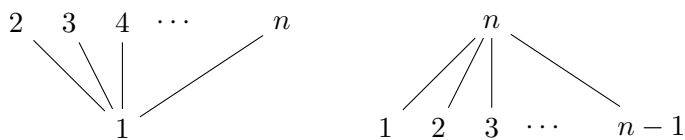
\begin{center}
    \begin{tikzpicture}[scale=.7]
      \node (1) at (-2,2) {$2$};
      \node (2) at (-1,2) {$3$};
      \node (3) at ( 0,2) {$4$};
      \node (n) at ( 3,2) {$n$};
      \node (t) at ( 0,0) {$1$};
      \node at (1,2) {$\cdots$};
      \draw (1) -- (t);
      \draw (2) -- (t);
      \draw (3) -- (t);
      \draw (n) -- (t);
    \end{tikzpicture}
    \qquad
    \begin{tikzpicture}[scale=.7]
      \node (1) at (-2,0) {$1$};
      \node (2) at (-1,0) {$2$};
      \node (3) at ( 0,0) {$3$};
      \node (n) at ( 3,0) {$n-1$};
      \node (t) at ( 0,2) {$n$};
      \node at (1,0) {$\cdots$};
      \draw (1) -- (t);
      \draw (2) -- (t);
      \draw (3) -- (t);
      \draw (n) -- (t);
    \end{tikzpicture}
  \end{center}
  have~$2n$ and, respectively, $4(n-1)$ facets by construction.  The short
  rooted tree has $2^n$ vertices in accordance with~\Cref{prop:rooted_cube}.
  These are given by $\1_F$ and $\1_F+\1_{[n]}$ for any subset $F \subseteq
  \{2,\ldots,n\}$.  On the other hand, the short hanging tree has $2^{n-1} +
  2$ vertices.  These are given by $\1_F$ for any filter~$F \subseteq [n]$,
  and $\1_{[n]} + \1_{\{n\}}$.
\end{example}

A full-dimensional polytope $\Po \subset \R^n$ is \Defn{$\boldsymbol 2$-level}
if for every facet $F \subset \Po$ there is a unique $\q \in \R^n$ such that
the two parallel hyperplanes $\aff(F)$ and $\q + \aff(F)$ contain all vertices
of~$\Po$. $2$-level polytopes enjoy many favorable geometric properties; see,
for example,~\cite{Twisted,GPT,WSZ}.  As a corollary to the proof of
\Cref{prop:vertices}, we note the following.

\begin{cor}
    Lipschitz polytopes are $2$-level.
\end{cor}

In particular, Sullivant~\cite{Sullivant} showed that every pulling
triangulation of a $2$-level polytope is unimodular. In the following section,
we describe a unimodular triangulation of $\Lip(\P)$ that is not of
pulling-type.  It might be interesting to study the combinatorial implications
of pulling triangulations of Lipschitz polytopes.

\section{Triangulations and volumes}%
\label{sec:triangulations}%

By definition, $\Lip(\P)$ is an \emph{alcoved polytope} in the sense of
Lam--Postnikov~\cite{LP}. This implies, that $\Lip(\P)$ comes with a regular
and unimodular triangulation and we will make use of this triangulation
to compute the volumes and $h^*$-vectors of Lipschitz polytopes.  We
briefly recap the setup. One considers the affine braid arrangement in
$\R^{n+1}$ given by the hyperplanes
\[
    \{ \x \in \R^{n+1} \newmid x_i - x_j = \beta \} 
    \qquad \text{ for $0 \le i < j \le n$ and $\beta \in \Z$. }
\]
The lineality space of the braid arrangement is $\R\cdot \1$ and the
restriction to $\{ x_0 = 0\} \cong \R^n$ yields an essential arrangement
$\widetilde{\mathcal{A}}_n$ that decomposes $\R^n$ into infinitely bounded
regions.  Whenever convenient, we also refer to the $0$-th coordinate of a
point $\p \in \R^n$ which, by construction, satisfies $p_0 = 0$.

Let $\p \in \R^n$ be a point not contained in any of the hyperplanes of
$\widetilde{\mathcal{A}}_n$.  Then $p_i \not\in \Z$ for all $i$ and hence $\p
= \q + \r$ with $\q \in \Z^n$ and $\r \in (0,1)^n$. Moreover, $p_i - p_j
\not\in \Z$ implies $r_i \neq r_j$ for $i \neq j$ and there is a unique
permutation $\tau \in \SymGrp_n$ such that 
\[
    0 \ < \
    r_{\tauinv(1)} \ < \
    r_{\tauinv(2)} \ < \
    \cdots \ < \
    r_{\tauinv(n)} \ < \ 1 \,,
\]
where we write $\tauinv \eqdef \tau^{-1}$ for the inverse permutation.
Define
\[
    \Delta_\tau \ \eqdef \ \big\{ \x \in \R^n \newmid 0 \ \le \ x_{\tauinv(1)}  \ \le  \
    x_{\tauinv(2)}  \ \le \
    \cdots  \ \le  \
    x_{\tauinv(n)} \  \le \  1 \big\}.
\]
This is an $n$-dimensional simplex with integral vertices and volume
$\frac{1}{n!}$, i.e.\ $\Delta_\tau$ is a \Defn{unimodular simplex} with
respect to the lattice $\Z^n$.  The closed region containing $\p$ is given by
$\q + \Delta_\tau$ and since $\p$ was arbitrary, all closed regions of
$\tilde{\mathcal{A}}_n$ are of that  form. In particular, any two regions are
isomorphic with respect to a $\Z^n$-preserving affine transformation.  The
closed regions of $\tilde{\mathcal{A}}_n$ are called \Defn{alcoves} and a
convex polytope $\Po \subset \R^n$ is \Defn{alcoved} if it is the union of
alcoves. Thus, an alcoved polytope naturally comes with a 
triangulation into unimodular simplices. Since the triangulation is induced by
a hyperplane arrangement, it is regular.

We call a finite poset $(\P,\preceq)$ \Defn{naturally labeled} if we identify
its ground set with $\{ 1,\dots,n \}$ for $n = |\P|$ such that $a \prec b$
implies $a < b$. In this case, we can set $\Phat := \P \cup \{\Pbot := 0\}$
and with this labelling, it is clear that the Lipschitz polytope as defined
in~\eqref{eqn:Lip} is an alcoved polytope. We now determine the alcoves that
compose $\Lip(\P)$.

To later simplify notations, we set $\tau(0) \eqdef 0$ for any permutation
$\tau \in \SymGrp_n$ and we recall that $q_0 = 0$ for any $\q \in \R^n$.  We
moreover define the \Defn{descent set} 
\[
    \Des(\tau) \ \eqdef \ \big\{ 0 \le i  < n \newmid \tau(i) > \tau(i+1) \big\}
\]
and the \Defn{inverse descent set} by $\iDes(\tau) \eqdef \Des(\tauinv)$.
Thus, $i \in \iDes(\tau)$  if and only if $i$ and $i+1$ are out of order in
the one-line notation of~$\tau$.  We also set $\des(\tau)\eqdef|\Des(\tau)|$
and $\ides(\tau) \eqdef |\iDes(\tau)|$.  In particular, $\tau(0) = 0 <
\tau(1)$ so that~$0$ is never a descent or inverse descent. Any ordered subset
$S = \{ s_1 \prec s_2 \prec \cdots \prec s_k \} \subseteq \Phat$ determines a
subword of $\tau$
\[
    \tau|_S \ \eqdef \ \tau(s_1)\tau(s_2)\dots \tau(s_k) \,
\]
and we define $\des(\tau|_S)$ as the number descents of the word $\tau|_S$.

\newcommand\DC{\mathrm{DC}}%
\newcommand\desP{\des_{\P,\tau}}%

We call a permutation $\tau \in \SymGrp_n$ \Defn{descent-compatible} with $\P$
if for all $a \succ \Pbot$ and any saturated chain $C = \{ \Pbot = c_1 \cov
c_2 \cov \cdots \cov c_k = a \} \subseteq \Phat$, the number of descents
$\des(\tau|_C)$ is independent of $C$.  Of course, it suffices to require this
for all $a \in \max(\P)$.  We denote by $\DC(\P) \subseteq \SymGrp_n$ the
descent-compatible permutations of $\P$. For $\tau \in \DC(\P)$ and $a \in P$,
we write $\desP(a)$ for the number of descents of $\des(\tau|_C)$ for any
maximal chain $C$ ending in $a$.

\begin{thm}\label{thm:triang}
    Let $(\P,\preceq)$ be a naturally labeled poset. Then $\q + \Delta_\tau
    \subseteq \Lip(\P)$ for $\tau \in \SymGrp_n$ and $\q 
    \in \Z^n$ if and only if $\tau$ is descent-compatible with $\P$ and $q_a =
    \desP(a)$ for all $a \in \P$.
\end{thm}

\begin{proof}
    Observe that $\q + \Delta_\tau$ is part of the alcoved triangulation of an
    alcoved polytope $\Po$ if and only if $\q + \c \in \Po$, where $\c$ is any
    point in the relative interior of $\Delta_\tau$. A canonical choice is the
    barycenter $\c^\tau$ of $\Delta_\tau$ given by $c^\tau_i  \eqdef
    \frac{\tau(i)}{n+1}$ for $i =1,\dots,n$.  Hence, we need to determine when
    $\q + \c^\tau$ satisfies the inequalities given in~\eqref{eqn:Lip}.  Now,
    if $a \in \P$ is a minimum, then 
    \[
        0 \ \le \ (\q + \c^\tau)_a \eq q_a + \tfrac{\tau(a)}{n+1} \ \le \ 1
    \]
    and hence $q_a = 0$. For a cover relation $a \cov b$, we calculate
    \[
        0 \ \le \ q_b - q_a + \tfrac{\tau(b) - \tau(a)}{n+1} \ \le \ 1\,.
    \]
    If $\tau(b) > \tau(a)$, then this holds if $q_b = q_a$. If $\tau(b) <
    \tau(a)$, then $q_b = q_a + 1$. Thus, $q_b$ is the number of descents of
    $\des(\tau|_C)$ where $C = \{ \Pbot = c_1 \cov c_2 \cov \cdots \cov c_k =
    b \}$ is any saturated chain.
\end{proof}

Note that if $\P = \{1,\dots,n\}$ is totally ordered, then $\DC(\P) =
\SymGrp_n$ and $\desP(i)$ is the number of descents in the word $\tau(1)
\cdots \tau(i)$.  

\begin{prop}
\label{prop:rootedtrees}
    Let $\P$ be a connected poset on $n$ elements. Then $\DC(\P) =
    \SymGrp_n$ if and only if $\P$ is a rooted tree.
\end{prop}
\begin{proof}
    If $\P$ is a rooted tree, then there is a unique saturated chain from
    $\Pbot$ to any given $b \in \P$ and hence any $\tau \in \SymGrp_n$
    is trivially descent-compatible.
    Conversely, if there are two
    distinct maximal chains $C_1,C_2$ ending in $b \in \P$, then it is easy to
    find a permutation $\tau \in \SymGrp_n$ with $\des(\tau|_{C_1}) \neq
    \des(\tau|_{C_2})$.
\end{proof}

\begin{example}
\label{ex:shorttrees2}
  We have already seen in the previous proposition that all $n!$ permutations
  are descent compatible for the short rooted tree from \Cref{ex:shorttrees}.
  On the other hand, the descent compatible permutations for the short hanging
  tree are those $2(n-1)!$ permutations of $[n]$ for which $\tau(n) \in\{ 1,n\}$.
\end{example}

If $\Po \subset \R^n$ is a full-dimensional lattice polytope, the Ehrhart
function $E(\Po,k) \eqdef |k \Po \cap \Z^n|$ agrees with a polynomial in $k$
of degree $n = \dim \Po$ and the \Defn{$\boldsymbol h^*$-polynomial} of~$\Po$
is defined by
\[
    h^*(\Po,z) \eq h^*_0 + h^*_1 z + \cdots + h^*_n z^d  \ \eqdef 
    (1-z)^{n+1}
    \sum_{k \ge 0} E(\Po,k) z^k \, .
\]
See, for example,~\cite{BeckRobins} for details.
If $\Po$ has a unimodular triangulation, then the $h^*$-polynomial can be
computed very elegantly by means of \emph{half-open decompositions}. Let $\Po
= \Po_1 \cup \cdots \cup \Po_r$ be a dissection of~$\Po$ into (lattice)
polytopes, i.e., every $\Po_i$ is a full-dimensional (lattice) polytope that
does not meet the interior of~$\Po_j$ for $j \neq i$. A point $\w \in
\Int(\Po)$ is in general position with respect to the dissection if $\w$ is
not contained in the arrangement of facet-defining hyperplanes of~$\Po_i$ for
all $i$. The point $\w$ is \Defn{beyond} a facet $F \subset \Po_i$ if the
facet-defining hyperplane $\aff(F)$ separates $\w$ from the interior of
$\Po_i$. The \Defn{half-open} polytope associated to $\Po_i$ and $\w$ is
\[
    \hopen_\w \Po_i \ := \ \Po_i \setminus \bigcup_F F \, ,
\]
where the union is over all facets $F \subset \Po_i$ for which $\w$ is beyond.

\begin{lem}[{\cite[Thm.~3]{KV}}]\label{lem:HO}
    Let $\Po = \Po_1 \cup \cdots \cup \Po_m$ be a dissection and $\w \in
    \Int(\Po)$ in general position with respect to the dissection. Then
    \[
        \Po \ = \ 
        \hopen_\w \Po_1 \uplus
        \hopen_\w \Po_2 \uplus
        \cdots \uplus
        \hopen_\w \Po_m \, .
    \]
    In particular, if all polytopes $\Po_i$ are lattice polytopes, then
    \[
        h^*(\Po,z) \ = \ 
        h^*(\hopen_\w \Po_1,z) +
        h^*(\hopen_\w \Po_2,z) +
        \cdots +
        h^*(\hopen_\w \Po_m,z) \, .
    \]
\end{lem}

Since alcoved polytopes are invariant under lattice translations and
coordinate permutations, we choose a suitable embedding before computing the
$h^*$-vector.

\begin{prop}
    Let $\Po \subset \R^n$ be a full-dimensional alcoved polytope. Then there
    is a lattice translation and a relabeling of coordinates such that
    $\Delta_{\rm{id}} \subseteq \Po \subset \R^n_{\ge0}$.
\end{prop}
\begin{proof}
    Let $\ell(\x)$ be any linear function such that $\ell(\p) > 0$ for all $\p
    \in \R^n_{\ge 0}$ with $\p \neq 0$. Then $\q$ is the unique minimizer of
    $\ell(\x)$ over $\q + \Delta_\tau$ for all $\q$ and $\tau$. Since $\Po$ can be dissected into
    finitely many alcoves, this shows that there is a unique $\q \in \Po$ that
    minimizes $\ell(\x)$ over $\Po$. Since this holds for all such linear
    functions, this shows that there is a point $\q_0 \in \Po$ that minimizes
    all $\ell(\x)$ and hence $\Po \subseteq \q + \R^n_{\ge 0}$. 
    Thus, $\Delta_\tau \subseteq \Po - \q \subset \R^n_{\ge 0}$ for some
    $\tau$ and  relabeling the coordinates finishes the proof.
\end{proof}

Making use of the previous result, we may compute the half-open decomposition
with respect to the point $\w = \frac{1}{n+1}(1,2,\dots,n)$.  With these
conventions, we are  now ready to compute the $h^*$-polynomial of Lipschitz
polytopes. For a naturally labeled poset $\P$ and a descent-compatible
permutation $\tau$, let us define $\stat_\P(\tau)$ as the number of $i \in
\{0,\ldots,n-1\}$ such that
\begin{align}\label{eq.defstatAlt}
  \desP(\tauinv(i))  <   \desP(\tauinv(i+1)) &\quad \text{ or } \quad
  \big( \desP(\tauinv(i)) = \desP(\tauinv(i+1)) \text{ and }i \in
  \iDes(\tau) \big)\\
\intertext{or, equivalently, as the number of pairs $(a,b) \in \Phat \times \Phat$ with $\tau(a)
= \tau(b) - 1$ and}
  \desP(a) \ < \ \desP(b) &\quad \text{ or } \quad
  \bigl( \desP(a) \ = \ \desP(b) \quad \text{and} \quad a > b
  \bigr)\,.
\label{eq.defstat}
\end{align}
The equivalence of the two descriptions is easily seen by considering $i = \tau(a)$ and $i+1 = \tau(b)$.

As $\desP$ is by definition weakly increasing along chains in~$\P$, the
subsets of~$\P$ for which $\desP$ is constant partitions $\P$ into
\emph{layers}. The description of $\stat_\P(\tau)$ can then be easily read off
the poset labeled by~$\tau$. The following example illustrates this.
\begin{example}
  We consider the following naturally labeled poset $\P$ on $\{1,\ldots,9\}$,
  and the permutation $\tau = 423716598 \in \DC(\P)$.  The image of $\tau$ is
  given in big (black) while the natural labelling is given in small (blue).
  Moreover, we indicated the (boundaries of the) layers of $\q^\tau =
  (0,1,1,1,2,1,1,2,2)$ for $q^\tau_a = \desP(a)$ in red.

  \begin{center}
    \begin{tikzpicture}[scale=1]
      \node[label={ [label={[xshift=0.15cm, yshift=-0.5cm]\color{blue}\tiny$1$}]}] (4) at ( 0,0) {$4$};
      \node[label={ [label={[xshift=0.15cm, yshift=-0.5cm]\color{blue}\tiny$2$}]}] (2) at (-1,1) {$2$};
      \node[label={ [label={[xshift=0.15cm, yshift=-0.5cm]\color{blue}\tiny$3$}]}] (3) at ( 1,1) {$3$};
      \node[label={ [label={[xshift=0.15cm, yshift=-0.5cm]\color{blue}\tiny$6$}]}] (6) at (-2,2) {$6$};
      \node[label={ [label={[xshift=0.15cm, yshift=-0.5cm]\color{blue}\tiny$5$}]}] (1) at ( 0,2) {$1$};
      \node[label={ [label={[xshift=0.15cm, yshift=-0.5cm]\color{blue}\tiny$4$}]}] (7) at ( 2,2) {$7$};
      \node[label={ [label={[xshift=0.15cm, yshift=-0.5cm]\color{blue}\tiny$7$}]}] (5) at (-1,3) {$5$};
      \node[label={ [label={[xshift=0.15cm, yshift=-0.5cm]\color{blue}\tiny$8$}]}] (9) at ( 1,3) {$9$};
      \node[label={ [label={[xshift=0.15cm, yshift=-0.5cm]\color{blue}\tiny$9$}]}] (8) at ( 0,4) {$8$};
      \draw (4) -- (2) -- (6) -- (5) -- (8);
      \draw (4) -- (3) -- (7) -- (9) -- (8);
      \draw (2) -- (1) -- (5);
      \draw (3) -- (1);

      \draw [red, thick] plot [smooth] coordinates {(-1,0) (0,.9) (1,0)};
      \draw [red, thick] plot [smooth] coordinates {(-3,2) (-2,3) (-1,2) (0,1.1) (1,2) (0,3) (1,4)};
    \end{tikzpicture}
  \end{center}
  Finally,
  \begin{align*}
    \big| \{ i &\newmid \desP(\tauinv(i)) < \desP(\tauinv(i+1)) \}\big| \ = \ \big| \{4,7 \} \big| \ = \ 2, \\
    \big| \{ i &\newmid \desP(\tauinv(i)) = \desP(\tauinv(i+1)), \  i \in \iDes(\tau) \}\big| \ = \ \big| \{6 \} \big| \ = \ 1,
  \end{align*}
  or, equivalently,
  \begin{align*}
    \big| \{(a,b) &\newmid \desP(a) < \desP(b), \quad
    \tau(a) = \tau(b)-1 \}\big| \ = \ \big| \{(1,7), (4,9) \}\big| \ = \ 2, \\
    \big| \{(a,b) &\newmid \desP(a) = \desP(b), \quad
    \tau(a) = \tau(b)-1, \quad a > b \}\big| \ = \ \big| \{(6,4) \}\big| \ = \ 1.
  \end{align*}
  We thus obtain $\stat_\P(\tau) = 3$.
\end{example}

\begin{thm}\label{thm:Lip-Hstar}
    Let $(\P,\preceq)$ be a naturally labeled poset.  Then
    \[
      h^*\big(\Lip(\P),z\big) \ = \ \GenFun_{\stat}(\P,z) \ \eqdef \ \sum_{\tau \in
      \DC(\P)} z^{\stat_P(\tau)} \,.
    \]
\end{thm}

In order to prove \Cref{thm:Lip-Hstar}, we extract the main technical tool
into the following lemma.

\begin{lem}\label{lem:HO-alcove}
    Let $\tau \in \SymGrp_n$ and $\q = (q_1,\ldots,q_n) \in \Z^n_{\ge 0}$.
    Then
    \[
        h^*(\hopen_\w(\q + \Delta_\tau),z)  \ = \ z^{|A(\tau,\q)| +
        |D(\tau,\q)|} \, ,
    \]
    where
    \begin{align*}
        A(\tau,\q) \ \eqdef \ \big\{ 0 \le i < n : \tauinv(i) < \tauinv(i+1)
        \quad \text{ and } \quad & q_{\tauinv(i)} < q_{\tauinv(i+1)} \big\}
        \ \text{
        and }\\
        D(\tau,\q) \ \eqdef \ \big\{ 0 \le i < n : \tauinv(i) > \tauinv(i+1)
        \quad \text{ and } \quad & q_{\tauinv(i)} \le q_{\tauinv(i+1)} \big\}\,.
    \end{align*}
\end{lem}

\begin{proof}
    We determine the linear inequalities of $\q + \Delta_{\tau}$ that are
    violated for $\w$. First observe that 
    \[
        (\w - \q)_{\tauinv(n)}  \eq\tfrac{\tauinv(n)}{n+1} -
        q_{\tauinv(n)} \ \le \ 1
    \]
    since $q_j \ge 0$ for all $j$. Hence, the last inequality of $
    \Delta_{\tau}$ is always satisfied for $\w-\q$. For $0 \le i < n$,
    we have
    \[
        (\w - \q)_{\tauinv(i)} >
        (\w - \q)_{\tauinv(i+1)}
        \quad \Longleftrightarrow \quad
        q_{\tauinv(i+1)} +
        \underbrace{\tfrac{\tauinv(i)-\tauinv(i+1)}{n+1}}_{=:\delta}
        \ > \
        q_{\tauinv(i)}.
    \]
    If $i$ is a descent of $\tauinv$, then $1 > \delta > 0$ and the inequality
    holds if and only if $q_{\tauinv(i+1)} \ge q_{\tauinv(i)}$. Otherwise,
    $i$ is an ascent of $\tauinv$ and $-1 < \delta < 0$ and the inequality
    holds if and only if $q_{\tauinv(i+1)} > q_{\tauinv(i)}$.
\end{proof}

\begin{proof}[Proof of \Cref{thm:Lip-Hstar}]
    By \Cref{thm:triang}, the simplices of the alcove triangulation of
    $\Lip(\P)$ are given by $\q^\tau + \Delta_\tau$ for $\tau \in \DC(\P)$ and
    $q^\tau_a = \desP(a)$ for all $a \in \P$. In particular, if $\P$
    is naturally labeled, then $\Delta_{\mathrm{id}} \subseteq \Lip(\P)
    \subseteq \R^n_{\ge0}$ and we can apply \Cref{lem:HO-alcove} to
    $\hopen_\w(\q^\tau + \Delta_\tau)$ and it is easy to see that
    $\stat_\P(\tau) = |A(\tau,\q^\tau)| + |D(\tau,\q^\tau)|$.  \Cref{lem:HO}
    then completes the proof.
\end{proof}

If $\P$ is a chain on $n$ elements, then $\DC(\P) = \SymGrp_n$ and the
statistic given in \Cref{thm:Lip-Hstar} reduces to a known permutation
statistic. For a permutation $\tau \in \SymGrp_n$, a \Defn{big ascent} (or \Defn{$\boldsymbol
2$-ascent}) is an index $0 \le i < n$ such that $\tau(i+1)
- \tau(i) \ge 2$. In particular, $i=0$ is a big ascent if and only if $\tau(1)
> 1$. We record the number of big ascents by $\bigasc(\tau)$, and set $\bigiasc(\tau) \eqdef \bigasc(\tauinv)$ and
\[
    A^{(2)}_{n}(z) \ = \ \sum_{i=0}^n A^{(2)}(n,i) z^i \ \eqdef \ \sum_{\tau \in
    \SymGrp_n} z^{\bigiasc(\tau)} \ = \ \sum_{\tau \in
    \SymGrp_n} z^{\bigasc(\tau)}\,.
\]

Big ascents and big descents appeared in the literature before, we refer to~\cite{MR3096140} for several identities involving big descents,
compare also~\cite[Lem.~6.5]{NanLi} for later reference.

\begin{thm}\label{thm:LipLin}
    Let $\P$ be a rooted tree on $n$ elements. Then
    \[
        \GenFun_{\stat}(\P,z) \ = \ A^{(2)}_n(z) \, .
   \] 
\end{thm}

\begin{proof}
    We prove the result only for the case that $\P$ is the chain on~$n$
    elements. By \Cref{prop:rooted_cube}, $\Lip(\P)$ and $\Lip(\P')$ are
    lattice-equivalent whenever $\P$ and $\P'$ are rooted trees on the same
    number of elements and hence $\GenFun_{\stat}(\P,z) =
    \GenFun_{\stat}(\P',z)$.

    From~\Cref{thm:triang} we infer that for any permutation $\tau \in
    \SymGrp_n$, $\q^\tau + \Delta_\tau \subset \Lip(\P)$ for $q^\tau_a =
    \desP(a)$.  We use \Cref{lem:HO-alcove} for the computation.
    Since $P = \{1,\dots,n\}$ is a chain, we observe that that $q^\tau_i \leq
    q^\tau_j$ for $i < j$ and $q_i = q_j$ if and only if there is no descent
    in $\tau(i)\tau(i+1)\dots\tau(j-1)$. 
    
    Now, if $1 \le i < n$ so that $\tauinv(i+1) < \tauinv(i)$, then
    $q^\tau_{\tauinv(i+1)} < q^\tau_{\tauinv(i)}$. Indeed, if $\tauinv(i+1) <
    \tauinv(i)$, then $\tau(\tauinv(i+1)) \tau(\tauinv(i+1)+1) \cdots \tau(\tauinv(i))$ inevitably
    contains a descent. This implies that $D(\tau,\q^\tau) = \emptyset$.  
    
    Otherwise, $\tauinv(i+1) > \tauinv(i)$ and hence $q^\tau_{\tauinv(i)} \leq
    q^\tau_{\tauinv(i+1)}$ with equality if and only if we have $\Des(\tau) \cap
    \{\tauinv(i),\ldots,\tauinv(i+1)-1\} = \emptyset$.  But this is the case
    if and only if $\tauinv(i+1) = \tauinv(i)+1$.  Therefore,
    $q^\tau_{\tauinv(i)} < q^\tau_{\tauinv(i+1)}$ if and only if $\tauinv(i+1)
    \geq \tauinv(i) + 2$.
    This gives $\stat_\P(\tau) = \bigiasc(\tau)$ in the case of the $n$-chain, which yields the statement.
\end{proof}

In light of \Cref{prop:rooted_cube} and the fact that the coefficients of the
$h^*$-polynomial of the cube are given by Eulerian numbers, \Cref{thm:LipLin} implies the
following corollary.

\begin{cor}
\label{cor:Hstarcube}
  Let $\P$ be a rooted tree on $n$ elements. Then
  \[
      \GenFun_{\stat}(\P,z) = A^{(2)}_n(z) = \sum_{\tau \in \SymGrp_n} z^{\ides(\tau)}.
  \] 
\end{cor}

We illustrate \Cref{thm:Lip-Hstar} with the computation of the
$h^*$-polynomial for the Lipschitz polytope of short rooted and hanging trees.

\begin{example}\label{ex:ides}
    Using the description of $\stat_\P$ given in~\eqref{eq.defstatAlt} in the
    case of the short rooted and hanging trees from \Cref{ex:shorttrees}, the
    following yields
    \[
      \stat_\P(\tau) = \ides(\tau) \quad \text{for} \quad \tau \in \DC(\P)
    \]
    in these two cases.  For the short rooted tree we have
  \begin{align*}
    \big\{ i &\newmid \desP(\tauinv(i)) < \desP(\tauinv(i+1)) \big\} &= \quad
      \begin{cases}
        \emptyset\hspace*{88pt} &\text{ if } \tau(1) = 1 \\
        \{ 0 \}   &\text{ if } \tau(1) > 1
      \end{cases}, \\
    \big\{ i &\newmid \desP(a) = \desP(b) \quad i \in \iDes(\tau) \big\} &= \quad 
        \begin{cases} 
          \iDes(\tau) &\text{ if } \tau(1) = 1 \\
          \iDes(\tau) \setminus \{ \tau(1)-1 \}   &\text{ if } \tau(1) > 1
        \end{cases}.
  \end{align*}
  On the other hand, for the short hanging tree, we have $\tau(n) \in \{1,n\}$ and two analogous considerations for these two possibilities.
\end{example}

Even though one easily finds examples of posets $\P$ and permutations $\tau
\in \DC(\P)$ for which $\stat_\P(\tau) \neq \ides(\tau)$, the \emph{generating function search functionality} of the Combinatorial Statistic Finder
\url{www.FindStat.org}~\cite{FS2017} suggests the following conjectural
generalization of \Cref{cor:Hstarcube} to all posets, in agreement with \Cref{ex:ides}.

\begin{conj}
\label{con:findstat}
  Let $\P$ be a poset.  Then
    \[
      \GenFun_{\stat}(\P,z) \ = \ \sum_{\tau \in
      \DC(\P)} z^{\ides(\tau)} \,.
    \]
\end{conj}

\section{Ranked posets}\label{sec:ranked}

A poset~$\P$ is \Defn{ranked}, if for all $a,b \in \P$, every maximal chain in
$\Pint{a,b} = \{ c \in \P : a \preceq c \preceq b\}$ has the same length. If
$\P$ has a minimal element, then this is equivalent to the existence of a
\Defn{rank function} $\rho : \P \to \Znn$ such that $\rho(\Pbot) = 0$ and
$\rho(b) = \rho(a) + 1$ for $a \cov b$. 

\begin{prop}\label{prop:CS}
    If $\P$ is a poset such that $\Phat$ is ranked, then $\Lip(\P)$ is
    centrally-symmetric with respect to $\rho$, i.e., $\rho - \Lip(\P) =
    \Lip(\P)$.
\end{prop}
\begin{proof}
    Simply note that $\rho(b) - \rho(a) - (f(b) - f(a)) = 1 - (f(b) - f(a))$
    for all $\Pbot \preceq a \cov b$.
\end{proof}

A lattice polytope $\Po \subset \R^n$ is \Defn{reflexive} if $\0$ is the only
lattice point in the interior and the polar polytope is again a lattice
polytope. A lattice polytope $\Po$ is \Defn{$\boldsymbol r$-Gorenstein} for some $r \in
\Z_{>0}$ if $r\Po$ contains a unique lattice point $\q$ and $r\Po-\q$ is
reflexive.

\begin{prop}\label{prop:Gorenstein}
    If $\P$ is a poset such that $\Phat$ is ranked, then $\Lip(\P)$ is
    $2$-Gorenstein.
\end{prop}
\begin{proof}
    From~\eqref{eqn:Lip}, it is clear that $f \in \Z^\P$ is in the interior of
    $2\Lip(\P)$ if $f(a) = 1$ for all $a \in \min(\P)$ and $f(b) =
    f(a) + 1$ for all $a \cov b$. Hence, $f$ is a rank function, which is
    unique whenever $\Phat$ is ranked. The polytope $2\Lip(\P) - \rho$ is
    given by $-1 \le f(b) - f(a) \le 1$ for all $\Pbot \preceq a \cov b$,
    which implies that the polar polytope is a lattice polytope with vertices
    $\pm \1_{a}$ and $\pm(\1_{b} - \1_{a})$.
\end{proof}

Stanley~\cite{CCA} noted that, in the case that an $n$-dimensional polytope $\Po$ is
$r$-Gorenstein, one has $h^*_i(\Po)  =  h^*_{n+1-r-i}(\Po)$ for all~$r$.
Bruns--R\"omer~\cite{BR} showed that if an $r$-Gorenstein polytope has a
regular and unimodular triangulation, then the $h^*$-vector is unimodal. For
Lipschitz polytopes, such a triangulation is vouched for by \Cref{thm:triang}
and with \Cref{prop:Gorenstein}, we get the following.

\begin{thm}\label{thm:Lip_Goren}
    Let $\P$ be a poset on $n$ elements such that $\Phat$ is ranked and let
    \[
        \GenFun_{\stat}(\P,z) \ = \  
        g_0(\P)  + g_1(\P) z + \cdots + g_{n-1}(\P) z^{n-1} \ = \ 
        \sum_{\tau \in \DC(\P)} z^{\stat_P(\tau)} \,.
    \]
    Then 
    \[
        g_i(\P) \eq g_{n-1-i}(\P) \quad \text{ and } \quad
        g_i(\P) \ \le \ g_{i+1}(\P) \text{ for }i \le \lfloor\tfrac{n-1}{2}\rfloor.
    \]
\end{thm}

It should be noted that if $\Phat$ is ranked, then $\DC(\P)$ is invariant
under reversals, \ie, if $\tau \in \DC(\P)$ then $\overline\tau \in \DC(\P)$
where $\overline\tau(i) := n+1 - \tau(i)$.  However, this symmetry does not
induce the symmetry of the statistic.

\begin{openproblem}
    Give a combinatorial proof of \Cref{thm:Lip_Goren}.
\end{openproblem}

For posets such that $\Phat$ is ranked, the Lipschitz polytope $\Lip(\P)$ can
also be constructed in a different way, that explains the central-symmetry as
well as the Gorenstein property. If $\Phat$ is ranked, then the order cone
$\OK(\P)$ is a Gorenstein cone, that is, there is a unique point $\q \in
\OK(\P)$ such that $\Int(\OK(\P)) \cap \Z^\P = (\q + \OK(\P)) \cap \Z^\P$.
Hibi~\cite{OrdGoren} showed that $\OK(\P)$ is Gorenstein if and only if
$\Phat$ is ranked and hence $\q = \rho$.  We thus get the following
description of $\Lip(\P)$.

\begin{cor}
  Let $\P$ be a poset such that $\Phat$ is ranked with rank function~$\rho$.
  Then
  \[
      \Lip(\P) \ = \ \OK(\P) \cap ( - \OK(\P) + \rho)\,.
  \]
\end{cor}

In this case $\Lip(\P)$ is a \emph{spindle} in the sense of
Santos~\cite{Santos}.

\section{\texorpdfstring{$\P$}{P}-Hypersimplices}\label{sec:Hyp}
\newcommand\height{\mathrm{ht}}

Let $(\P,\preceq)$ be a poset with a unique maximal element $\Ptop$.
Let $\height(\P)$ be the number of elements in a maximal chain in $\P$, the
\Defn{height} of~$\P$.  For $1 \le k \le \height(\P)$, we define the \Defn{$\boldsymbol{(\P,k)}$-hypersimplex} as 
\[
    \Delta(\P,k) \ \eqdef \ \big\{ f \in \Lip(\P) \newmid k-1 \le f(\Ptop)  \le
    k \big\}\,.
\]
The $(\P,k)$-hypersimplices are again lattice polytopes and, in fact, alcoved
and we can refine the description of vertices of \Cref{prop:vertices}.

\begin{prop}
    The vertices of $\Delta(\P,k)$ correspond exactly to neighbor-closed
    chains of non\-empty filters in $\P$ of length $k-1$ or $k$.
\end{prop}

In particular, this gives a nice interpretation of the first
$(\P,k)$-hypersimplex.

\begin{cor}\label{cor:Hyp_OP}
    For every poset $\P$ with~$\Ptop$, $\Delta(\P,1)$ is the order polytope $\OP(\P)$.
    Moreover, if $\Phat$ is ranked, then 
    \[
        \Delta(\P,k) \ \cong \ \Delta(\P,\height(\P)+1-k) \,,
    \]
    for $1 \le k \le \height(\P)$.
\end{cor}

In the case that $\P$ is the $n$-chain, $\Delta(\P,k)$ recovers, after the
transformation used already in \Cref{ex:trivial}, the well-known
\Defn{$\boldsymbol{(n,k)}$-hypersimplex}, introduced in~\cite{GGMS},
\begin{align*}
    \Delta(n,k) & \ \eqdef \  \conv\big\{ \y \in \{0,1\}^{n+1} \newmid y_1 + \cdots +
    y_{n+1} = k \big\} \\ 
    &\ \phantom{:}\cong \ \hspace*{24pt} \big\{ \y \in [0,1]^n \newmid k-1 \le y_1 + \cdots +
    y_{n} \le k \big\} \,.
\end{align*}

We can define the notion of \Defn{$\boldsymbol \P$-descents} on
$\DC(\P)$ by $\des_{\P}(\tau) := \desP(\Ptop)$ for $\tau \in
\DC(\P)$.

\begin{cor}\label{cor:P-descents}
    Let $\P$ be a poset on $n$ elements with~$\Ptop$.
    Then, for $1 \le k \le \height(\P)$,
    \[
        n! \cdot \vol \Delta(\P,k) \ = \ \big|\{ \tau \in \DC(\P) :
        \des_{\P}(\tau) = k-1 \}\big| \, .
    \]
\end{cor}

\Cref{cor:P-descents} is a is a direct consequence of \Cref{thm:triang} and
generalizes of the well-known result for $\P$ being an $n$-chain: The volume
of $\Delta(n+1,k+1)$ normalized by $n!$ is the Eulerian number $A(n,k)$.  This
result is attributed to Laplace and was proved by geometric means
in~\cite{StanleyHyp}; see also~\cite{LP, NanLi}.

The central-symmetry of $\Lip(\P)$ in the case that $\Phat$ is ranked
(\Cref{prop:CS}) , yields the following symmetry of $P$-descents, generalizing
the classical symmetry of the descent statistic.

\begin{cor}
    Let $\P$ be a poset with $\Ptop$ such that $\Phat$ is ranked. Then for $1
    \le k \le \height(\P)$
    \[
        \big|\{ \tau \in \DC(\P) : \des_{\P}(\tau) = k-1 \} \big|
    \ = \
        \big|\{ \tau \in \DC(\P) : \des_{\P}(\tau) = \height(\P) - k \} \big| \,.
    \]
\end{cor}

The short hanging trees of \Cref{ex:shorttrees} are exactly the connected
posets with $\Ptop$ of height $2$. \Cref{cor:Hyp_OP} yields that in this case
$\Lip(\P)$ can be decomposed into two congruent copies of $\OP(\P)$. The
following combinatorial implication can be seen also directly from
\Cref{ex:ides}.

\begin{cor}
    Let $(\P,\preceq)$ be a short hanging tree on $n$ elements. Then
    \[
        h^*(\Lip(\P),z) \ = \ (1+z) \cdot h^*(\OP(\P),z) \, .
    \]
\end{cor}

Li~\cite{NanLi} gave an interpretation of the $h^*$-vector of the half-open
hypersimplices
\[
    \Deltaprime(n,k+1)    \ := \ \big\{\y \in [0,1]^n \newmid k < y_1 + \cdots +
    y_{n} \le k+1  \big\}  \,,
\]
for $1 \le k < n$. Using generating functions and shellings of the alcoved
triangulation, the following was shown.

\begin{thm}[{\cite[Thm.~1.3]{NanLi}}]\label{cor:nan}
  \[
    h^*(\Deltaprime(n,k+1),z)\ = \ \sum_{\substack{\tau \in \SymGrp_{n} \\ \des(\tau)=k}} z^{\bigiasc(\tau)}.
  \]
\end{thm}

Since the alcoved triangulation of $\Lip(\P)$ is compatible with restriction
to the half-open $(\P,k)$-hypersimplices
\[
    \Deltaprime(\P,k+1) \ := \ \big\{ f \in \Lip(\P) \newmid k < f(\Ptop) \le
    k+1 \big\} \, ,
\]
we obtain the following generalization of \Cref{cor:nan}, which is the case
that $\P$ is the $n$-chain.

\begin{cor}
\label{cor:HstarPkslide}
    Let $\P$ be a finite poset and $1 \le k \le h(\P)$.
    Then
    \[
      h^*\big(\Deltaprime(\P,k),z\big) \ =\ \sum_{\substack{\tau \in \DC(\P)\\\des(\tau)=k}}
      z^{\stat_P(\tau)} \,.
    \]
\end{cor}

\bibliographystyle{siam}
\bibliography{LipschitzPolytopes}

\end{document}